\documentclass{amsart}
\usepackage{amssymb}
\usepackage{amsmath}
\usepackage{amsfonts}

\newtheorem{theorem}{Theorem}
\theoremstyle{plain}
\newtheorem{acknowledgement}{Acknowledgement}

\newtheorem{condition}{Condition}

\newtheorem{definition}{Definition}

\newtheorem{lemma}{Lemma}

\newtheorem{proposition}{Proposition}
\newtheorem{remark}{Remark}

\numberwithin{equation}{section}
\input{tcilatex}

\begin{document}
\title[Non-degeneracy of Wiener functionals arising from RDEs]{%
Non-degeneracy of Wiener functionals arising from rough differential
equations }
\author{Thomas Cass, Peter Friz, Nicolas Victoir}

\begin{abstract}
Malliavin Calculus is about Sobolev-type regularity of functionals on Wiener
space, the main example being the It\^{o} map obtained by solving stochastic
differential equations. Rough path analysis is about strong regularity of
solution to (possibly stochastic) differential equations. We combine
arguments of both theories and discuss existence of a density for solutions
to stochastic differential equations driven by a general class of
non-degenerate Gaussian processes, including processes with sample path
regularity worse than Brownian motion.
\end{abstract}

\maketitle

\section{Introduction}

It is basic question in probability theory whether a given stochastic
process $\left\{ Y_{t}:t\geq 0\right\} $ with values in some Euclidean space
admits, at fixed positive times, a density with respect to Lebesgue measure.
In a non-degenerate Markovian setting - ellipticity of the generator - an
affirmative answer can be given using Weyl's lemma as discussed in McKean's
classical 1969 text \cite{McK}. Around the same time, H\"{o}rmander's
seminal work on hypoelliptic partial differential operators enabled
probabilists to obtain criteria for existence (and smoothness) for densities
of certain degenerate diffusions. This dependence on the theory of partial
differential equations was removed when P. Malliavin devised a purely
probabilistic approach, perfectly adapted to prove existence and smoothness
of densities.

We recall some key ingredients of Malliavin's machinery, known as \textit{%
Malliavin Calculus} or \textit{stochastic calculus of variations}. Most of
it can formulated in the setting of an abstract Wiener spaces $\left( W,%
\mathcal{H},\mu \right) $. The concept is standard \cite{Be, Ma, Nu06, Sh}
as is the notion of a \textit{weakly non-degenerate} $\mathbb{R}^{e}$-valued
Wiener functional $\varphi $ which has the desirable property that the image
measure $\varphi _{\ast }\mu $ is absolutely continuous with respect to
Lebesgue measure on $\mathbb{R}^{e}$. (Functionals which are \textit{%
non-degenerate }have a smooth density.) Precise definitions are given later
on in the text.

Given these abstract tools, we turn to the standard Wiener space $%
C_{0}\left( \left[ 0,1\right] ,\mathbb{R}^{d}\right) $ equipped with Wiener
measure.\ From It\^{o}'s theory, we can realize diffusion processes by
solving the stochastic differential equation%
\begin{equation*}
dY=\sum_{i=1}^{d}V_{i}\left( Y\right) \circ dB^{i}+V_{0}\left( Y\right)
dt\equiv V\left( Y\right) \circ dB+V_{0}\left( Y\right) dt,\,\,\,Y\left(
0\right) =y_{0}\in \mathbb{R}^{e}
\end{equation*}%
driven by a $d$-dimensional Brownian motion $B$ along sufficiently
well-behaved (drift - resp. diffusion) vector fields $V_{0},...,V_{d}$. The
It\^{o}-map $B\mapsto Y$ is notorious for its lack of strong regularity
properties which rules out the use of any Fr\'{e}chet-calculus. On the
positive side, it is smooth in a Sobolev sense on Wiener space ("\textit{%
smooth in the sense of Malliavin}"). \ Under condition%
\begin{equation*}
(E):\text{ span}\left[ V_{1},...,V_{d}\right] _{y_{0}}=\mathcal{T}_{y_{0}}%
\mathbb{R}^{e}\cong \mathbb{R}^{e};
\end{equation*}%
or a less restrictive H\"{o}rmander's condition (allowing all Lie brackets
of $V_{0},...,V_{d}$ in spanning $\mathcal{T}_{y_{0}}\mathbb{R}^{e}$) the
solution map $B\mapsto Y_{t}$, for $t>0$, is non-degenerate and one reaches
the desired conclusion that $Y_{t}$ has a (smooth) density. This line of
reasoning due to P. Malliavin provides a direct probabilistic approach to
the study of transition densities and has found applications from stochastic
fluid dynamics to interest rate theory. It also shows that the Markovian
structure is not essential and one can, for instance, adapt these ideas to
study densities of It\^{o}-diffusions as was done by Kusuoka and Stroock in 
\cite{KS87}.

Our interest lies in stochastic differential equations of type%
\begin{equation}
dY=V\left( Y\right) dX+V_{0}\left( Y\right) dt  \label{GaussRDE}
\end{equation}

where $X$ is a multi-dimensional Gaussian process. Such differential
equations arise, for instance, in financial mathematics \cite{BaHoSi, Gu06,
Gu07} or as model for studying ergodic properties of non-Markovian systems 
\cite{Ha05}. Assuming momentarily enough sample path regularity so that (\ref%
{GaussRDE}) makes sense path-by-path by Riemann-Stieltjes (or Young)
integration the question whether or not $Y_{t}$ admits a density is
important and far from obvious. To the best of our knowledge, all results in
that direction are restricted to fractional Brownian motion with Hurst
parameter $H>1/2$. Existence of a density for $Y_{t}$, $t>0$, was
established in \cite{NuSa} under condition (E). Smoothness results then
appeared in \cite{NuHu}, and under H\"{o}rmander's condition in \cite{BH}.

The purpose of this paper is to give a first demonstration of the powerful
interplay between Malliavin calculus and rough path analysis. After some
remarks on $\mathcal{H}$-differentiability and a representation of the
Malliavin covariance in terms of a 2D Young integral, we show that weak
non-degeneracy (in the sense of Malliavin) of solutions to (\ref{GaussRDE})
at times $T>0$, and hence existence of a density, remains valid in an almost
generic sense. Our assumptions are

\begin{itemize}
\item The vector fields $V=\left( V_{1},...,V_{d}\right) $ at $y_{0}$ span $%
\mathcal{T}_{y_{0}}\mathbb{R}$, i.e. condition (E).

\item The continuous, centered Gaussian driving signal $X$ is such that the
stochastic differential equation (\ref{GaussRDE}) makes sense as \textit{%
rough differential equation} (RDE), \cite{L98, LQ02}. Applied in our
context, this represents a unified framework which covers \textit{at once}
Gaussian signals with nice sample paths (such as fBM with $H>1/2\,$),
Brownian motion, Gaussian (semi-)martingales, and last not least Gaussian
signals with sample path regularity worse than Brownian motion provided
there exists a sufficiently regular stochastic area\footnote{%
For orientation, fractional Brownian motion is covered for any $H>1/3$.}.

\item The Gaussian driving signal is sufficiently non-degenerate which is
clearly needed to rule our examples such as $X\equiv 0$ or the Brownian
bridge $X_{t}=B_{t}\left( \omega \right) -\left( t/T\right) B_{T}\left(
\omega \right) $.
\end{itemize}

Smoothness of densities remains an open problem, some technical remarks
about the difficulties involved are found in the last section.

\begin{acknowledgement}
We would like to thank James Norris for related discussions and an anonymous
referee whose comments prompted a clarification in our discussion of $%
\mathcal{H}$-regularity. The second author is partially supported by a
Leverhulme Research Fellowship.
\end{acknowledgement}

\section{Preliminaries on ODE and RDEs}

\subsection{Controlled ordinary differential equations}

Consider the ordinary differential equations, driven by a smooth $\mathbb{R}%
^{d}$-valued signal $f=f\left( t\right) $ along sufficiently smooth and
bounded vector fields $V=\left( V_{1},...,V_{d}\right) $ and a drift vector
field $V_{0}$%
\begin{equation*}
dy=V(y)df+V_{0}\left( y\right) dt,\,\,\,y\left( t_{0}\right) =y_{0}\in 
\mathbb{R}^{e}.
\end{equation*}%
We call $U_{t\leftarrow t_{0}}^{f}\left( y_{0}\right) \equiv y_{t}$ the
associated flow. Let $J$ denote the Jacobian of $U$. It satisfies the ODE
obtain by formal differentiation w.r.t. $y_{0}$. More specifically, 
\begin{equation*}
a\mapsto \left\{ \frac{d}{d\varepsilon }U_{t\leftarrow t_{0}}^{f}\left(
y_{0}+\varepsilon a\right) \right\} _{\varepsilon =0}
\end{equation*}%
is a linear map from $\mathbb{R}^{e}\rightarrow \mathbb{R}^{e}$ and we let $%
J_{t\leftarrow t_{0}}^{f}\left( y_{0}\right) $ denote the corresponding $%
e\times e$ matrix. It is immediate to see that 
\begin{equation*}
\frac{d}{dt}J_{t\leftarrow t_{0}}^{f}\left( y_{0}\right) =\left[ \frac{d}{dt}%
M^{f}\left( U_{t\leftarrow t_{0}}^{f}\left( y_{0}\right) ,t\right) \right]
\cdot J_{t\leftarrow t_{0}}^{f}\left( y_{0}\right)
\end{equation*}%
where $\cdot $ denotes matrix multiplication and%
\begin{equation*}
\frac{d}{dt}M^{f}\left( y,t\right) =\sum_{i=1}^{d}V_{i}^{\prime }\left(
y\right) \frac{d}{dt}f_{t}^{i}+V_{0}^{\prime }\left( y\right) .
\end{equation*}
Note that $J_{t_{2}\leftarrow t_{0}}^{f}=J_{t_{2}\leftarrow t_{1}}^{f}\cdot
J_{t_{1}\leftarrow t_{0}}^{f}.$ We can consider Gateaux derivatives in the
driving signal and define%
\begin{equation*}
D_{h}U_{t\leftarrow 0}^{f}=\left\{ \frac{d}{d\varepsilon }U_{t\leftarrow
0}^{f+\varepsilon h}\right\} _{\varepsilon =0}.
\end{equation*}%
One sees that $D_{h}U_{t\leftarrow 0}^{f}$ satisfies a linear ODE and the
variation of constants formula leads to%
\begin{equation*}
D_{h}U_{t\leftarrow 0}^{f}\left( y_{0}\right)
=\int_{0}^{t}\sum_{i=1}^{d}J_{t\leftarrow s}^{f}\left( V_{i}\left(
U_{s\leftarrow 0}^{f}\right) \right) dh_{s}^{i}.
\end{equation*}

\subsection{Rough differential equations\label{SectionRDEs}}

Let $p\in \left( 2,3\right) $. A geometric $p$-rough path $\mathbf{x}$ over $%
\mathbb{R}^{d}$ is a continuous path on $\left[ 0,T\right] $ with values in $%
G^{2}\left( \mathbb{R}^{d}\right) $, the step-$2$ nilpotent group over $%
\mathbb{R}^{d}$, of finite $p$-variation relative to $d$, the
Carnot-Caratheodory metric on $G^{2}\left( \mathbb{R}^{d}\right) $, i.e.%
\begin{equation*}
\sup_{n}\sup_{t_{1}<...<t_{n}}\sum_{i}d\left( \mathbf{x}_{t_{i}},\mathbf{x}%
_{t_{i+1}}\right) ^{p}<\infty .
\end{equation*}%
Following \cite{L98, FV1, FV2, FVbook} we realize $G^{2}\left( \mathbb{R}%
^{d}\right) $ as the set of all $\left( a,b\right) \in $ $\mathbb{R}%
^{d}\oplus \mathbb{R}^{d\times d}$ for which $Sym\left( b\right) \equiv
a^{\otimes 2}/2$. (This point of view is natural: a smooth $\mathbb{R}^{d}$%
-valued path $x=\left( x_{t}^{i}\right) _{i=1,\dots ,d}$, enhanced with its
iterated integrals $\int_{0}^{t}\int_{0}^{s}dx_{u}^{i}dx_{s}^{j}$, gives
canonically rise to a $G^{2}\left( \mathbb{R}^{d}\right) $-valued path.)
Given $\left( a,b\right) \in G^{2}\left( \mathbb{R}^{d}\right) $ one gets
rid of the redundant $Sym\left( b\right) $ by $\left( a,b\right) \mapsto $ $%
\left( a,b-a^{\otimes }/2\right) \in \mathbb{R}^{d}\oplus so\left( d\right) $%
. Applied to $x$ enhanced with its iterated integrals over $\left[ 0,t\right]
$ this amounts to look at the path $x$ and its (signed) areas $%
\int_{0}^{t}\left( x_{s}^{i}-x_{0}^{i}\right) dx_{s}^{j}-\int_{0}^{t}\left(
x_{s}^{j}-x_{0}^{j}\right) dx_{s}^{i},\,\,i,j\in \left\{ 1,\dots ,d\right\} .
$ Without going into too much detail, the group structure on $G^{2}\left( 
\mathbb{R}^{d}\right) $ can be identified with the (truncated) tensor
multiplication and is relevant as it allows to relate algebraically the path
and area increments over adjacent intervals; the mapping $\left( a,b\right)
\mapsto $ $\left( a,b-a^{\otimes }/2\right) $ maps the Lie group $%
G^{2}\left( \mathbb{R}^{d}\right) $ to its Lie algebra; at last, the
Carnot-Caratheodory metric is defined intrinsically as (left-)invariant
metric on $G^{2}\left( \mathbb{R}^{d}\right) $ and satisfies $\left\vert
a\right\vert +\left\vert b\right\vert ^{1/2}\lesssim d\left( \left(
0,0\right) ,\left( a,b\right) \right) \lesssim \left\vert a\right\vert
+\left\vert b\right\vert ^{1/2}$.

One can then think of a geometric $p$-rough path $\mathbf{x}$ as a path $x:%
\left[ 0,T\right] \rightarrow \mathbb{R}^{d}$ enhanced with its iterated
integrals (equivalently: area integrals) although the later need not make
classical sense. For instance, \textit{almost every}\ joint realization of
Brownian motion and L\'{e}vy's area process is a geometric $p$-rough path 
\cite{LQ02, FV1, FLS}. The Lyons theory of rough paths \cite{LQ98, LQ02,
FVbook} then gives deterministic meaning to the rough differential equation
(RDE)%
\begin{equation*}
dy=V\left( y\right) d\mathbf{x}
\end{equation*}%
for $\mathrm{Lip}^{\gamma }$-vector fields (in the sense of Stein), $\gamma
>p$. By considering the space-time rough path $\mathbf{\tilde{x}}=\left( t,%
\mathbf{x}\right) $ and $\tilde{V}=\left( V_{0},V_{1},...,V_{d}\right) $
this form is general enough to cover differential equations with drift%
\footnote{%
(1) If $\mathbf{x}$ lifts a $\mathbb{R}^{d}$-valued path $x$, then $\mathbf{%
\tilde{x}}$ is construct with cross-integrals of type $\int xdt,\int tdx$
all of which are canonically defined. (2) Including $V_{0}$ in the
collection $\tilde{V}$ leads to suboptimal regularity requirements for $V_{0}
$ which could be avoided by a direct analysis.}. The solution induces a flow 
$y_{0}\mapsto U_{t\leftarrow t_{0}}^{\mathbf{x}}\left( y_{0}\right) $. The
Jacobian $J_{t\leftarrow t_{0}}^{\mathbf{x}}$ of the flow exists and
satisfies a linear RDE, as does the directional derivative%
\begin{equation*}
D_{h}U_{t\leftarrow 0}^{\mathbf{X}}=\left\{ \frac{d}{d\varepsilon }%
U_{t\leftarrow 0}^{T_{\varepsilon h}\mathbf{x}}\right\} _{\varepsilon =0}
\end{equation*}%
for a smooth path $h$. If $\mathbf{x}$ arises from a smooth path $x$
together with its iterated integrals the \textit{translated rough path} $%
T_{h}\mathbf{x}$ is nothing but $x+h$ together with its iterated integrals.
In the general case, we assume $h\in C^{q\text{-var}}$ with $1/p+1/q>1$, the
translation $T_{h}\mathbf{x}$ can be written in terms of $\mathbf{x}$ and
cross-integrals between $\pi _{1}\left( \mathbf{x}_{0,\cdot }\right) =:x$
and the perturbation $h$. (These integrals are well-defined Young-integrals.)

\begin{proposition}
\label{VarOfConst}Let $\mathbf{X}$ be a geometric $p$-rough paths over $%
\mathbb{R}^{d}$ and $h\in C^{q}\left( \left[ 0,1\right] ,\mathbb{R}%
^{d}\right) $ such that $1/p+1/q>1$. Then%
\begin{equation*}
D_{h}U_{t\leftarrow 0}^{\mathbf{X}}\left( y_{0}\right)
=\int_{0}^{t}\sum_{i=1}^{d}J_{t\leftarrow s}^{\mathbf{X}}\left( V_{i}\left(
U_{s\leftarrow 0}^{\mathbf{X}}\right) \right) dh_{s}^{i}
\end{equation*}%
where the right hand side is well-defined as Young integral.
\end{proposition}

\begin{proof}
At least when $\gamma >p+1$, both $J_{t\leftarrow 0}^{\mathbf{X}}$ and $%
D_{h}U_{t\leftarrow 0}^{\mathbf{X}}$ satisfy (jointly with $U_{t\leftarrow
0}^{\mathbf{X}}$) a RDE driven by $\mathbf{X}$. This is an application of
Lyons' limit theorem and discussed in detail in \cite{LQ97, LQ98}. A little
care is needed since the resulting vector fields now have linear growth. It
suffices to rule out explosion so that the problem can be localized. The
needed remark is that $J_{t\leftarrow 0}^{\mathbf{X}}$ also satisfy a linear
RDE of form 
\begin{equation*}
dJ_{t\leftarrow 0}^{\mathbf{X}}=d\mathbf{M}^{\mathbf{X}}\cdot J_{t\leftarrow
0}^{\mathbf{X}}\left( y_{0}\right)
\end{equation*}%
where $d\mathbf{M}^{\mathbf{X}}=V^{\prime }\left( U_{t\leftarrow 0}^{\mathbf{%
X}}\left( y_{0}\right) \right) d\mathbf{X}$. Explosion can be ruled out by
direct iterative expansion and estimates of the Einstein sum as in \cite{L98}%
.
\end{proof}

\section{RDEs driven by Gaussian signals}

We consider a continuous, centered Gaussian process with independent
components $X=\left( X^{1},...,X^{d}\right) $ started at zero. This gives
rise to an abstract Wiener space $\left( W,\mathcal{H},\mu \right) $ where $%
W=\mathcal{\bar{H}}\subset C_{0}\left( \left[ 0,T\right] ,\mathbb{R}%
^{d}\right) $. Note that $\mathcal{H}=\oplus _{i=1}^{d}\mathcal{H}^{\left(
i\right) }$ and recall that elements of $\mathcal{H}$ are of form $h_{t}=%
\mathbb{E}\left( X_{t}\xi \left( h\right) \right) $ where $\xi \left(
h\right) $ is a Gaussian random variable. The ("reproducing kernel")
Hilbert-structure on $\mathcal{H}$ is given by $\,\left\langle h,h^{\prime
}\right\rangle _{\mathcal{H}}:=\mathbb{E}\left( \xi \left( h\right) \xi
\left( h^{\prime }\right) \right) $.

Existence of a Gaussian geometric $p$-rough path above $X$ is tantamount to
the existence of certain L\'{e}vy area integrals. From the point of view of
Stieltjes integration the existence of L\'{e}vy's area for Brownian motion
is a miracle. In fact, a subtle cancellation due to orthogonality of
increments of Brownian motion is responsible for convergence and this
suggests that processes with sufficiently fast decorrelation of their
increments will also give rise to a stochastic L\'{e}vy area. The resulting
technical conditions appears in \cite{LQ02} for instance. For Gaussian
processes, a cleaner (and slightly weaker) condition can be given in terms
of the \textit{2D\ variation} properties of the covariance function $R\left(
s,t\right) =\mathbb{E}\left( X_{s}\otimes X_{t}\right) =diag\left( R^{\left(
1\right) },...,R^{\left( k\right) }\right) $.The assumption, writing $%
X_{t,t^{\prime }}:=X_{t^{\prime }}-X_{t}$,%
\begin{equation*}
\left\vert R\right\vert _{\rho \text{-var;}\left[ 0,T\right] ^{2}}^{\rho
}:=\sup_{D=\left( t_{i}\right) }\sum_{i,j}\left\vert \mathbb{E}\left(
X_{t_{i},t_{i+1}}X_{t_{j},t_{j+1}}\right) \right\vert ^{\rho }<\infty 
\end{equation*}%
for $\rho <2$ is known \cite{FV07} to be sufficient (and essentially
necessary) for the existence of a natural lift of $X$ to a geometric $p$%
-rough path $\mathbf{X}$ for any $p>2\rho $. Observe that the covariance of
Brownian motion has finite $\rho $-variation with $\rho =1$. As a more
general example, a direct computation shows that fractional Brownian motion
has finite $\rho $-variation with $\rho =1/\left( 2H\right) $.

The assumption of $\left\vert R\right\vert _{\rho \text{-var}}<\infty $ has
other benefits, notably the following embedding theorem \cite{FV07}. Since
it is crucial for our purposes we repeat the short proof, as we apply it
componentwise we can assume $d=1$.

\begin{proposition}
\label{CM_pVar_embedding}Let $R$ be the covariance of a real-valued centered
Gaussian process. If $R$ is of finite $\rho $-variation, then $\mathcal{H}%
\hookrightarrow C^{\rho \text{-var}}$. More, precisely, for all $h\in 
\mathcal{H}$,%
\begin{equation*}
\left\vert h\right\vert _{\rho \text{$-var$;}\left[ s,t\right] }\leq \sqrt{%
\left\langle h,h\right\rangle _{\mathcal{H}}}\sqrt{R_{\rho \text{-var;}\left[
s,t\right] ^{2}}}.
\end{equation*}
\end{proposition}

\begin{proof}
Every element $h\in \mathcal{H}$ can be written as $h_{t}=\mathbb{E}\left(
ZX_{t}\right) $ for the Gaussian r.v. $Z=\xi \left( h\right) $. We may
assume that $\left\langle h,h\right\rangle _{\mathcal{H}}=\mathbb{E}\left(
Z_{.}^{2}\right) =1$. Let $\left( t_{j}\right) $ be a subdivision of $\left[
s,t\right] $ and write $\left\vert x\right\vert _{l^{r}}=\left(
\sum_{i}x_{i}^{r}\right) ^{1/r}$ for $r\geq 1.$ If $\rho ^{\prime }$ denote
the H\"{o}lder conjugate of $\rho $ we have%
\begin{eqnarray*}
&&\left( \sum_{j}\left\vert h_{t_{j},t_{j+1}}\right\vert ^{\rho }\right)
^{1/\rho }=\sup_{\beta ,\left\vert \beta \right\vert _{l^{\rho ^{\prime
}}}\leq 1}\sum_{j}\beta _{j}h_{t_{j},t_{j+1}}=\sup_{\beta ,\left\vert \beta
\right\vert _{l^{\rho ^{\prime }}}\leq 1}\mathbb{E}\left( Z\sum_{j}\beta
_{j}X_{t_{j},t_{j+1}}\right) \\
&\leq &\sup_{\beta ,\left\vert \beta \right\vert _{l^{\rho ^{\prime }}}\leq
1}\sqrt{\sum_{j,k}\beta _{j}\beta _{k}\mathbb{E}\left(
X_{t_{j},t_{j+1}}X_{t_{k},t_{k+1}}\right) } \\
&\leq &\sup_{\beta ,\left\vert \beta \right\vert _{l^{\rho ^{\prime }}}\leq
1}\sqrt{\left( \sum_{j,k}\left\vert \beta _{j}\right\vert ^{\rho ^{\prime
}}\left\vert \beta _{k}\right\vert ^{\rho ^{\prime }}\right) ^{\frac{1}{\rho
^{\prime }}}\left( \sum_{j,k}\left\vert \mathbb{E}\left(
X_{t_{j},t_{j+1}}X_{t_{k},t_{k+1}}\right) \right\vert ^{\rho }\right) ^{%
\frac{1}{\rho }}} \\
&\leq &\left( \sum_{j,k}\left\vert \mathbb{E}\left(
X_{t_{j},t_{j+1}}X_{t_{k},t_{k+1}}\right) \right\vert ^{\rho }\right)
^{1/\left( 2\rho \right) }\leq \sqrt{R_{\rho \text{-var;}\left[ s,t\right]
^{2}}}\text{.}
\end{eqnarray*}%
Optimizing over all subdivision $\left( t_{j}\right) $ of $\left[ s,t\right]
,$ we obtain our result.
\end{proof}

\bigskip

One observes that for Brownian motion ($\rho =1$) this embedding is sharp.
Furthermore, the rough path translation $T_{h}\mathbf{X}$, which involves
Young integrals, makes sense if $\rho <\rho ^{\ast }=3/2$. (This critical
value comes from $1/\rho +1/\left( 2p\right) \sim 1/\rho +1/\left( 2\rho
\right) =3/\left( 2\rho \right) $ and equating to $1$.) The moral is that
one can take deterministic directional derivatives in Cameron-Martin
directions as long as $\rho <3/2$. (En passant, we see that the effective
tangent space to Gaussian RDE solutions is strictly bigger than the usual
Cameron-Martin space as long as $\rho <3/2$. In the special case of
Stratonovich SDEs a related result predates rough path theory and goes back
to Kusuoka \cite{Ku}). For this reason we will assume%
\begin{equation*}
\rho <3/2
\end{equation*}%
for the remainder of this paper. This entails that we are dealing with
geometric $p$-rough paths $\mathbf{X}$ for which we may assume%
\begin{equation*}
p\in \left( 2\rho ,3\right) \text{.}
\end{equation*}

\begin{definition}
\cite{Ku82}, \cite[Section 4.1.3]{Nu06}, \cite[Section 3.3]{UZ00}Given an
abstract Wiener space $\left( W,\mathcal{H},\mu \right) $, a random variable
(i.e. measurable map) $F:W\rightarrow \mathbb{R}$ is continuously $\mathcal{H%
}$-differentiable, in symbols $F\in C_{\mathcal{H}}^{1}$, if for $\mu $%
-almost every $\omega $, the map%
\begin{equation*}
h\in \mathcal{H}\mapsto F\left( \omega +h\right)
\end{equation*}%
is continuously Fr\'{e}chet differentiable. A vector-valued r.v. $F=\left(
F^{1},...,F^{e}\right) :W\rightarrow \mathbb{R}^{e}$ is continuously $%
\mathcal{H}$-differentiable iff each $F^{i}$ is continuously $\mathcal{H}$%
-differentiable. In particular, $\mu $-almost surely, $DF\left( \omega
\right) =\left( DF^{1}\left( \omega \right) ,...,DF^{e}\left( \omega \right)
\right) $ is a linear bounded map from $\mathcal{H}\rightarrow \mathbb{R}%
^{e} $
\end{definition}

\begin{remark}
(1) The notion of continuous $\mathcal{H}$-differentiability was introduced
in \cite{Ku82} and plays a fundamental role in the study of transformation
of measure on Wiener space. Integrability properties of $F$ and $DF$ aside, $%
C_{\mathcal{H}}^{1}$-regularity is stronger than Malliavin differentiability
in the usual sense. Indeed, by \cite[Thm 4.1.3]{Nu06} (see also \cite{Ku82}, 
\cite[Section 3.3]{UZ00}) $C_{\mathcal{H}}^{1}$ implies $\mathbb{D}%
_{loc}^{1,2}$-regularity where the definition of $\mathbb{D}_{loc}^{1,2}$ is
based on the commonly used Shigekawa Sobolev space $\mathbb{D}^{1,p}$. (Our
notation here follows \cite[Sec. 1.2, 1.3.4]{Nu06}). This remark will be
important to us since it justifies the use of Bouleau-Hirsch's criterion
(e.g. \cite[Section 2.1.2]{Nu06}) for establishing absolute continuity of $F$
(cf. proof of theorem \ref{ThmEllDens}). \newline
(2) Although not relevant to the sequel of this paper, it is interesting to
compare $C_{\mathcal{H}}^{1}$-regularity with the Kusuoka-Stroock Sobolev
spaces. Following \cite{Ku82b, KS82} one defines $\mathbb{\tilde{D}}^{1,p}$%
as the space of random-variables $F$ which are (i) ray-absolutely-continuous
(RAC) in the sense that for every $h\in \mathcal{H}$ there is an absolutely
continuous version of the process $\left\{ F\left( \omega +th\right) :t\in 
\mathbb{R}\right\} $; (ii) stochastically Gateaux differentiable (SGD) in
the sense that there exists an $\mathcal{H}$-valued r.v. $\tilde{D}F$ such
that for every $h\in \mathcal{H}$,%
\begin{equation}
\left( F\left( \omega +th\right) -F\left( \omega \right) \right)
/t\rightarrow \left\langle \tilde{D}F,h\right\rangle _{\mathcal{H}}\text{ as 
}t\rightarrow 0  \label{SGD}
\end{equation}%
probability with respect to $\mu $; and (iii) such that $F\in L^{p}$ and $%
\tilde{D}F\in L^{p}\left( \mathcal{H}\right) $.\newline
From Sugita \cite{Su85} it is known that $\mathbb{\tilde{D}}^{1,p}=\mathbb{D}%
^{1,p}$, at least for $p\in \left( 1,\infty \right) $. Since $C_{\mathcal{H}%
}^{1}$-regularity is a local property it has nothing to say about the
integrability property (iii) but it does imply a fortiori the regularity
properties (i) and (ii). Indeed, (i) is trivially satisfied (without the
need of $h$-dependent modifications!). As for (ii), $\tilde{D}F$ is given by
the Fr\'{e}chet differential $DF\left( \omega \right) $ of $h\in \mathcal{H}%
\mapsto F\left( \omega +h\right) $ and the convergence (\ref{SGD}) holds not
only in probability but $\mu $-almost surely.
\end{remark}

\begin{proposition}
\label{GateauxFrechet}Let $\rho <3/2$. For fixed $t\geq 0$, the $\mathbb{R}%
^{e}$-valued random variable%
\begin{equation*}
\omega \mapsto U_{t\leftarrow 0}^{\mathbf{X}\left( \omega \right) }\left(
y_{0}\right)
\end{equation*}%
is continuously $\mathcal{H}$-differentiable.
\end{proposition}

\begin{proof}
Choose $p>2\rho $ such that $1/p+1/\rho >1$. We may assume that $\mathbf{X}%
\left( \omega \right) $ has been defined so that $\mathbf{X}\left( \omega
\right) $ is a geometric $p$-rough path for every $\omega \in W$. Let us
also recall for $h\in \mathcal{H}\subset C^{\rho \text{-var}}$, the
translation $T_{h}\mathbf{X}\left( \omega \right) $ can be written (for $%
\omega $ fixed!) in terms of $\mathbf{X}\left( \omega \right) $ and
cross-integrals between $\pi _{1}\left( \mathbf{X}_{0,\cdot }\right) =:X\in
C^{p\text{-var}}$ and $h$. (These integrals are well-defined
Young-integrals.) Thanks to the definition of $\mathbf{X}\left( \omega
\right) $ as the limit in probability of piecewise linear approximations to $%
X$ and its iterated integrals (cf. \cite{FV07}) and basic continuity
properties of Young integrals we see that the event%
\begin{equation}
\left\{ \omega :\mathbf{X}\left( \omega +h\right) \equiv T_{h}\mathbf{X}%
\left( \omega \right) \text{ for all }h\in \mathcal{H}\right\} 
\label{SetOfFullMeasureProofOfHdiff}
\end{equation}%
has probability one. We show that $h\in $ $\mathcal{H}\mapsto U_{t\leftarrow
0}^{\mathbf{X}\left( \omega +h\right) }\left( y_{0}\right) $ is continuously
Fr\'{e}chet differentiable for every $\omega $ in the above set of full
measure. By basic facts of Fr\'{e}chet theory, we must show (a) Gateaux
differentiability and (b) continuity of the Gateaux differential.\newline
Ad (a): Using $\mathbf{X}\left( \omega +g+h\right) \equiv T_{g}T_{h}\mathbf{X%
}\left( \omega \right) $ for $g,h\in \mathcal{H}$ it suffices to show
Gateaux differentiability of $U_{t\leftarrow 0}^{\mathbf{X}\left( \omega
+\cdot \right) }\left( y_{0}\right) $ at $0\in \mathcal{H}$. \ For fixed $t$%
, define 
\begin{equation*}
Z_{i,s}\equiv J_{t\leftarrow s}^{\mathbf{X}}\left( V_{i}\left(
U_{s\leftarrow 0}^{\mathbf{X}}\right) \right) .
\end{equation*}%
Note that $s\mapsto Z_{i,s}$ is of finite $p$-variation. We have, with
implicit summation over $i$, 
\begin{eqnarray*}
\left\vert D_{h}U_{t\leftarrow 0}^{\mathbf{X}}\left( y_{0}\right)
\right\vert  &=&\left\vert \int_{0}^{t}J_{t\leftarrow s}^{\mathbf{X}}\left(
V_{i}\left( U_{s\leftarrow 0}^{\mathbf{X}}\right) \right)
dh_{s}^{i}\right\vert  \\
&=&\left\vert \int_{0}^{t}Z_{i}dh^{i}\right\vert  \\
&\leq &c\left( \left\vert Z\right\vert _{p-var}+\left\vert Z\left( 0\right)
\right\vert \right) \times \left\vert h\right\vert _{\rho -var} \\
&\leq &c\left( \left\vert Z\right\vert _{p-var}+\left\vert Z\left( 0\right)
\right\vert \right) \times \left\vert h\right\vert _{\mathcal{H}}.
\end{eqnarray*}%
Hence, the linear map $DU_{t\leftarrow 0}^{\mathbf{X}}\left( y_{0}\right)
:h\mapsto D_{h}U_{t\leftarrow 0}^{\mathbf{X}}\left( y_{0}\right) \in \mathbb{%
R}^{e}$ is bounded and each component is an element of $\mathcal{H}^{\ast }$%
. We just showed that%
\begin{equation*}
h\mapsto \left\{ \frac{d}{d\varepsilon }U_{t\leftarrow 0}^{T_{\varepsilon h}%
\mathbf{X}\left( \omega \right) }\left( y_{0}\right) \right\} _{\varepsilon
=0}=\left\langle DU_{t\leftarrow 0}^{\mathbf{X}\left( \omega \right) }\left(
y_{0}\right) ,h\right\rangle _{\mathcal{H}}
\end{equation*}%
and hence%
\begin{equation*}
h\mapsto \left\{ \frac{d}{d\varepsilon }U_{t\leftarrow 0}^{\mathbf{X}\left(
\omega +\varepsilon h\right) }\left( y_{0}\right) \right\} _{\varepsilon
=0}=\left\langle DU_{t\leftarrow 0}^{\mathbf{X}\left( \omega \right) }\left(
y_{0}\right) ,h\right\rangle _{\mathcal{H}}
\end{equation*}%
emphasizing again that $\mathbf{X}\left( \omega +h\right) \equiv T_{h}%
\mathbf{X}\left( \omega \right) $ almost surely for all $h\in \mathcal{H}$
simultaneously. Repeating the argument with $T_{g}\mathbf{X}\left( \omega
\right) =\mathbf{X}\left( \omega +g\right) $ shows that the Gateaux
differential of $U_{t\leftarrow 0}^{\mathbf{X}\left( \omega +\cdot \right) }$
at $g\in \mathcal{H}$ is given by%
\begin{equation*}
DU_{t\leftarrow 0}^{\mathbf{X}\left( \omega +g\right) }=DU_{t\leftarrow
0}^{T_{g}\mathbf{X}\left( \omega \right) }.
\end{equation*}%
(b) It remains to be seen that $g\in \mathcal{H}\mapsto DU_{t\leftarrow
0}^{T_{g}\mathbf{X}\left( \omega \right) }\in L\left( \mathcal{H},\mathbb{R}%
^{e}\right) $, the space of linear bounded maps equipped with operator norm,
is continuous. To this end, assume $g_{n}\rightarrow _{n\rightarrow \infty }g
$ in $\mathcal{H}$ (and hence in $C^{\rho \text{-var}}$). Continuity
properties of the Young integral imply continuity of the translation
operator viewed as map $h\in C^{\rho \text{-var}}\mapsto T_{h}\mathbf{X}%
\left( \omega \right) $ as $p$-rough path (see \cite{LQ02}) and so 
\begin{equation*}
T_{g_{n}}\mathbf{X}\left( \omega \right) \rightarrow T_{g}\mathbf{X}\left(
\omega \right) 
\end{equation*}%
in $p$-variation rough path metric. To point here is that%
\begin{equation*}
\mathbf{x}\mapsto J_{t\leftarrow \cdot }^{\mathbf{x}}\text{ and }%
J_{t\leftarrow \cdot }^{\mathbf{x}}\left( V_{i}\left( U_{\cdot \leftarrow
0}^{\mathbf{x}}\right) \right) \in C^{p\text{-var}}
\end{equation*}%
depends continuously on $\mathbf{x}$ with respect to $p$-variation rough
path metric: using the fact that $J_{t\leftarrow \cdot }^{\mathbf{x}}$ and $%
U_{\cdot \leftarrow 0}^{\mathbf{x}}$ both satisfy rough differential
equations driven by $\mathbf{x}$ this is just a consequence of Lyons' limit
theorem (the \textit{universal limit theorem} of rough path theory). We
apply this with $\mathbf{x}=\mathbf{X}\left( \omega \right) $ where $\omega $
remains a fixed element in (\ref{SetOfFullMeasureProofOfHdiff}). It follows
that%
\begin{equation*}
\left\Vert DU_{t\leftarrow 0}^{T_{g_{n}}\mathbf{X}\left( \omega \right)
}-DU_{t\leftarrow 0}^{T_{g}\mathbf{X}\left( \omega \right) }\right\Vert
_{op}=\sup_{h:\left\vert h\right\vert _{\mathcal{H}}=1}\left\vert
D_{h}U_{t\leftarrow 0}^{T_{g_{n}}\mathbf{X}\left( \omega \right)
}-D_{h}U_{t\leftarrow 0}^{T_{g}\mathbf{X}\left( \omega \right) }\right\vert 
\end{equation*}%
and defining $Z_{i}^{g}\left( s\right) \equiv J_{t\leftarrow s}^{T_{g}%
\mathbf{X}\left( \omega \right) }\left( V_{i}\left( U_{s\leftarrow 0}^{T_{g}%
\mathbf{X}\left( \omega \right) }\right) \right) $, and similarly $%
Z_{i}^{g_{n}}\left( s\right) $, the same reasoning as in part (a) leads to
the estimate 
\begin{equation*}
\left\Vert DU_{t\leftarrow 0}^{T_{g_{n}}\mathbf{X}\left( \omega \right)
}-DU_{t\leftarrow 0}^{T_{g}\mathbf{X}\left( \omega \right) }\right\Vert
_{op}\leq c\left( \left\vert Z^{g_{n}}-Z^{g}\right\vert _{p\text{-var}%
}+\left\vert Z^{g_{n}}\left( 0\right) -Z^{g}(0)\right\vert \right) .
\end{equation*}%
From the explanations just given this tends to zero as $n\rightarrow \infty $
which establishes continuity of the Gateaux differential, as required, and
the proof is finished.
\end{proof}

\begin{definition}
\cite{Sh, Nu06, Ma} Given a continuously $\mathcal{H}$-differentiable%
\footnote{$D_{loc}^{1,p}$ is enough to guarantee the existence of an $%
\mathcal{H}$-valued derivative.} r.v. $F=\left( F^{1},...,F^{e}\right)
:W\rightarrow \mathbb{R}^{e}$ the Malliavin covariance matrix as the random
matrix given by%
\begin{equation*}
\sigma \left( \omega \right) :=\left( \left\langle
DF^{i},DF^{j}\right\rangle _{\mathcal{H}}\right) _{i,j=1,...,e}\in \mathbb{R}%
^{e\times e}.
\end{equation*}%
We call $F$ weakly non-degenerate if $\det \left( \sigma \right) \neq 0$
almost surely.
\end{definition}

We now give an integral representation of the Malliavin covariance matrix of 
$Y_{t}\equiv U_{t\leftarrow 0}^{\mathbf{X}\left( \omega \right) }\left(
y_{0}\right) $, the solution to the RDE driven by $\mathbf{X}\left( \omega
\right) $, the lift of $\left( X^{1},...,X^{d}\right) $, along vector fields 
$\left( V_{1},...,V_{d}\right) $, in terms of 2D\ Young integrals \cite{Yo,
To, FV07}.

\begin{proposition}
Let $\sigma _{t}=\left( \left\langle DY_{t}^{i},DY_{t}^{j}\right\rangle _{%
\mathcal{H}}:i,j=1,...,e\right) $ denote the Malliavin covariance matrix of $%
Y_{t}\equiv U_{t\leftarrow 0}^{\mathbf{X}\left( \omega \right) }\left(
y_{0}\right) $, the RDE solution of $dY=V\left( Y\right) d\mathbf{X}\left(
\omega \right) $. In the notation of section \ref{SectionRDEs} we have 
\begin{equation*}
\left( \left\langle DY_{t}^{i},DY_{t}^{j}\right\rangle _{\mathcal{H}}\right)
_{i,j=1,...,e}=\sum_{k=1}^{d}\int_{0}^{t}\int_{0}^{t}J_{t\leftarrow s}^{%
\mathbf{X}}\left( V_{k}\left( Y_{s}\right) \right) \otimes J_{t\leftarrow
s^{\prime }}^{\mathbf{X}}\left( V_{k}\left( Y_{s^{\prime }}\right) \right)
dR^{\left( k\right) }\left( s,s^{\prime }\right)
\end{equation*}
\end{proposition}

\begin{proof}
Let $\left( h_{n}^{\left( k\right) }:n\right) $ be an ONB\ of $\mathcal{H}%
^{\left( k\right) }$. It follows that $\left( h_{n}^{\left( k\right)
}:n=1,2,...;k=1,...,d\right) $ is an ONB\ of $\mathcal{H}=\oplus _{i=1}^{d}%
\mathcal{H}^{\left( k\right) }$ where we identify%
\begin{equation*}
h_{n}^{\left( 1\right) }\in \mathcal{H}^{\left( 1\right) }\equiv \left( 
\begin{array}{c}
h_{n}^{\left( 1\right) } \\ 
0 \\ 
... \\ 
0%
\end{array}%
\right) \in \mathcal{H}
\end{equation*}%
and similarly for $k=2,...,d$. From Parseval's identity,%
\begin{eqnarray*}
\sigma _{t} &=&\left( \left\langle DY_{t}^{i},DY_{t}^{j}\right\rangle _{%
\mathcal{H}}\right) _{i,j=1,...,e} \\
&=&\sum_{n,k}\left\langle DY_{t},h_{n}^{\left( k\right) }\right\rangle _{%
\mathcal{H}}\otimes \left\langle DY_{t},h_{n}^{\left( k\right)
}\right\rangle _{\mathcal{H}} \\
&=&\sum_{k}\sum_{n}\int_{0}^{t}J_{t\leftarrow s}^{\mathbf{X}}\left(
V_{k}\left( Y_{s}\right) \right) dh_{n,s}^{\left( k\right) }\otimes
\int_{0}^{t}J_{t\leftarrow s}^{\mathbf{X}}\left( V_{k}\left( Y_{s}\right)
\right) dh_{n,s}^{\left( k\right) } \\
&=&\sum_{k}\int_{0}^{t}\int_{0}^{t}J_{t\leftarrow s}^{\mathbf{X}}\left(
V_{k}\left( Y_{s}\right) \right) \otimes J_{t\leftarrow s^{\prime }}^{%
\mathbf{X}}\left( V_{k}\left( Y_{s^{\prime }}\right) \right) dR^{\left(
k\right) }\left( s,s^{\prime }\right) .
\end{eqnarray*}%
For the last step we used that%
\begin{equation*}
\sum_{n}\int_{0}^{T}fdh_{n}\int_{0}^{T}gdh_{n}=\int_{0}^{T}\int_{0}^{T}f%
\left( s\right) g\left( t\right) dR\left( s,t\right)
\end{equation*}%
whenever $f=f\left( t,\omega \right) $ and $g$ are such that the integrals
are a.s. well-defined Young-integrals. The proof is a consequence of $%
R\left( s,t\right) =\mathbb{E}\left( X_{s}X_{t}\right) $ and the $L^{2}$%
-expansion of the Gaussian process $X$,%
\begin{equation*}
X\left( t\right) =\sum_{n}\xi \left( h_{n}\right) h_{n}\left( t\right)
\end{equation*}%
where $\xi \left( h_{n}\right) $ form an IID family of standard Gaussians.
\end{proof}

\bigskip

Two special cases are worth considering: in the case of Brownian motion $%
dR\left( s,s^{\prime }\right) $ is a Dirac measure on the diagonal $\left\{
s=s^{\prime }\right\} $) and the double integral reduces to a (well-known)
single integral expression; in the case of fractional Brownian motion with $%
H>1/2$ it suffices to take the $\partial ^{2}/\left( \partial s\partial
t\right) $ derivative of $R_{H}\left( s,t\right) =\left(
t^{2H}+s^{2H}-\left\vert t-s\right\vert ^{2H}\right) /2$ to see that%
\begin{equation*}
dR_{H}\left( s,s^{\prime }\right) \sim \left\vert t-s\right\vert ^{2H-2}dsdt
\end{equation*}%
which is integrable iff $2H-2>-1$ or $H>1/2$. (The resulting double-integral
representation of the Malliavin covariance is also well-known and appears,
for instance, in \cite{NuSa, NuHu, BH}.)

\section{Existence of a density for Gaussian RDEs}

We remain in the framework of the previous sections, $Y_{t}\left( \omega
\right) \equiv U_{t\leftarrow 0}^{\mathbf{X}\left( \omega \right) }\left(
y_{0}\right) $ denotes the (random)\ RDE solution driven a Gaussian rough
path $\mathbf{X}$, the natural lift of a continuous, centered Gaussian
process with independent components $X=\left( X^{1},...,X^{d}\right) $
started at zero. Under the \textbf{standing assumption} of finite $\rho $%
-variation of the covariance, $\rho <3/2$, we know that $\mathbf{X}\left(
\omega \right) $ is a.s. a geometric $p$-rough path for $p\in \left( 2\rho
,3\right) $. Recall that this means that $\mathbf{X}$ can be viewed as path
in $G^{2}\left( \mathbb{R}^{d}\right) $, the step-$2$ nilpotent group over $%
\mathbb{R}^{d}$, of finite $p$-variation relative to the Carnot-Caratheodory
metric on $G^{2}\left( \mathbb{R}^{d}\right) $.

\begin{condition}
\label{EllipVF}\textbf{Ellipticity assumption on the vector fields:} The
vector fields $V_{1},...,V_{d}$ span the tangent space at $y_{0}$.
\end{condition}

\begin{condition}
\label{NonDegGauss}\textbf{Non-degeneracy of the Gaussian process on }$\left[
0,T\right] $\textbf{:} Fix $T>0$. We assume that for any smooth $f=\left(
f_{1},...,f_{d}\right) :\left[ 0,T\right] \rightarrow \mathbb{R}^{d}$%
\begin{equation*}
\left( \int_{0}^{T}fdh\equiv \sum_{k=1}^{d}\int_{0}^{T}f_{k}dh^{k}=0\forall
h\in \mathcal{H}\right) \implies f\equiv 0.
\end{equation*}
\end{condition}

Note that non-degeneracy on $\left[ 0,T\right] $ implies non-degeneracy on $%
\left[ 0,t\right] $ for any $t\in (0,T]$. It is instructive to see how this
condition rules out the Brownian bridge returning to the origin at time $T$
or earlier; a Brownian bridge which returns to zero after time $T$ is
allowed. The following lemma contains a few ramifications concerning
condition \ref{NonDegGauss}. Since $\mathcal{H}=\oplus _{k=1}^{d}\mathcal{H}%
^{\left( k\right) }$ there is no loss in generality in assuming $d=1$.

\begin{lemma}
\label{LemmaNDG}(i) The requirement that $f$ is smooth above can be relaxed
to $f\in C^{p\text{-var}}$ for $p>2\rho $ small enough.\newline
(ii) The requirement that $\int fdh=0\forall h\in \mathcal{H}$ can be
relaxed to the the quantifier "for all $h$ in some orthonormal basis of $%
\mathcal{H}$".\newline
(iii) The non-degeneracy condition \ref{NonDegGauss} is equivalent to saying
that for all smooth $f\neq 0$, the zero-mean Gaussian random variable $%
\int_{0}^{T}fdX$ (which exists as Young integral or via
integration-by-parts) has positive definite variance.)\newline
(iv) The non-degeneracy condition \ref{NonDegGauss} is equivalent to saying
that for all times $0<t_{1}<...<t_{n}<T$ the covariance matrix of $\left(
X_{t_{1}},...,X_{t_{n}}\right) $, that is,%
\begin{equation*}
\left( R\left( t_{i},t_{j}\right) \right) _{i,j=1,...,d}
\end{equation*}%
is (strictly) positive definite. \newline
\end{lemma}

\begin{proof}
(i) Continuity properties of the Young integral, noting that $1/p+1/\rho >1$
for $p>2\rho $ small enough, and weak compactness of the unit ball in $%
\mathcal{H}$. (ii) Any $h\in \mathcal{H}$ can be written as limit in $%
\mathcal{H}$ of $h^{\left[ n\right] }\equiv \sum_{k=1}^{n}\left\langle
h_{k},h\right\rangle h_{k}$ as $n\rightarrow \infty $ when $\left(
h_{k}\right) $ is an ONB\ for $\mathcal{H}$. From Proposition \ref%
{CM_pVar_embedding} it follows that $h$ is also the limit of $h^{\left[ n%
\right] }$ in $\rho $-variation topology. Conclude with continuity of the
Young integral. (iii) Every element $h\in \mathcal{H}$ can be written as $%
h_{t}=\mathbb{E}\left( ZX_{t}\right) $ with $Z\in \xi \left( \mathcal{H}%
\right) $. By taking $L^{2}$-limits one easily justifies the formal
computation%
\begin{equation*}
dh_{t}=\mathbb{E}\left( \dot{X}_{t}\,Z\right) dt\implies \int_{0}^{T}fdh=%
\mathbb{E}\left[ \left( \int_{0}^{T}fdX\right) \,Z\right] 
\end{equation*}%
and our condition is equivalent to saying that%
\begin{equation*}
\int_{0}^{T}fdX\perp \xi \left( \mathcal{H}\right) .
\end{equation*}%
On the other hand, it is clear that $\int fdX$ itself is an element of the
1st Wiener It\^{o} Chaos $\xi \left( \mathcal{H}\right) $ and so must be $0$
in $L^{2}$. In other words, saying that $\int fdh=0$ for all $h\in $ $%
\mathcal{H}$ says precisely that%
\begin{equation*}
\text{Var}\left[ \int_{0}^{T}fdX\right] =\left\vert
\int_{0}^{T}fdX\right\vert _{L^{2}}^{2}=0\text{.}
\end{equation*}%
Conversely, assume that%
\begin{equation*}
\text{Var}\left[ \int_{0}^{T}fdX\right] =0\text{.}
\end{equation*}%
Then $\int fdX=0$ with probability $1$ and by the Cameron-Martin theorem,
for any $h\in \mathcal{H}$,%
\begin{equation*}
\int_{0}^{T}fd\left( X+h\right) =0
\end{equation*}%
with probability one which implies $\int fdh=0$ for all $h\in $ $\mathcal{H}$%
. (iv) Suffices to note that%
\begin{equation*}
\int_{0}^{T}\int_{0}^{T}f\left( s\right) f\left( t\right) dR\left(
s,t\right) =\sum_{n}\left\vert \int_{0}^{T}f\left( s\right)
dh_{n}\right\vert ^{2}
\end{equation*}%
For the last step we used that%
\begin{equation*}
\sum_{n}\int_{0}^{T}fdh_{n}\int_{0}^{T}gdh_{n}=\int_{0}^{T}\int_{0}^{T}f%
\left( s\right) g\left( t\right) dR\left( s,t\right) 
\end{equation*}%
which is a consequence of $R\left( s,t\right) =\mathbb{E}\left(
X_{s}X_{t}\right) $ and the $L^{2}$-expansion $X\left( t\right) =\sum_{n}\xi
\left( h_{n}\right) h_{n}\left( t\right) $. (v) Left to the reader. (vi)
Obvious.
\end{proof}

\begin{remark}
The variance of $\int fdX$ can written as 2D Young integral,%
\begin{equation*}
\int_{\left[ 0,T\right] ^{2}}f_{s}f_{t}dR\left( s,t\right) .
\end{equation*}
\end{remark}

To put the following result in context, recall that the covariance of
Brownian motion, $\left( s,t\right) \mapsto \min \left( s,t\right) $, has
finite $1$-variation with $\rho =1$. For fractional Brownian motion with
Hurst parameter $H$ one can take $\rho =1/2H$. The following result then
applies to fractional Brownian driving signals with $H>1/3$.

\begin{theorem}
\label{ThmEllDens}Let $\mathbf{X}$ be natural lift of a continuous, centered
Gaussian process with independent components $X=\left(
X^{1},...,X^{d}\right) $, with finite $\rho $-variation of the covariance, $%
\rho <3/2$ and non-degenerate in the sense of condition \ref{NonDegGauss}.
Let $V=\left( V_{1},...,V_{d}\right) $ be a collection of $Lip^{3}$-vector
fields on $\mathbb{R}^{e}$ which satisfy the ellipticity condition \ref%
{EllipVF}. Then the solution to the (random) RDE%
\begin{equation*}
dY=V(Y)d\mathbf{X},\,Y\left( 0\right) =y_{0}\in \mathbb{R}^{e}
\end{equation*}%
admits a density at all times $t\in (0,T]$ with respect to Lebesgue measure
on $\mathbb{R}^{e}$.
\end{theorem}

\begin{proof}
Fix $t\in (0,T]$. From Proposition \ref{GateauxFrechet} we know that $%
U_{t\leftarrow 0}^{\mathbf{X}\left( \omega \right) }y_{0}=Y_{t}$ is
continuously $\mathcal{H}$-differentiable. By a well-known criterion due to
Bouleau-Hirsch\footnote{%
Combine the result of [Nualart, 4.1.3] and [Nualart, section 2].} the proof
is reduced to show a.s. invertibility of the Malliavin covariance matrix%
\begin{equation*}
\sigma _{t}=\left( \left\langle DY_{t}^{i},DY_{t}^{j}\right\rangle _{%
\mathcal{H}}\right) _{i,j=1,...,e}\in \mathbb{R}^{e\times e}.
\end{equation*}%
Assume there exists a (random) vector $v\in \mathbb{R}^{e}$ which
annihilates the quadratic form $\sigma _{t}$. Then\footnote{%
Upper $T$ denotes the transpose of a vector or matrix.}%
\begin{equation*}
0=v^{T}\sigma _{t}v=\left\vert \sum_{i=1}^{e}v_{i}DY_{t}^{i}\right\vert _{%
\mathcal{H}}^{2}\text{ \ and so\ }v^{T}DY_{t}\equiv
\sum_{i=1}^{e}v_{i}DY_{t}^{i}\in 0\in \mathcal{H}\text{.}
\end{equation*}%
By propositions \ref{VarOfConst} and \ref{GateauxFrechet}, 
\begin{equation}
\forall h\in \mathcal{H}:v^{T}D_{h}Y_{t}=\int_{0}^{t}\sum_{j=1}^{d}v^{T}J_{t%
\leftarrow s}^{\mathbf{X}}\left( V_{j}\left( Y_{s}\right) \right)
dh_{s}^{j}=0  \label{vDY_is_zero}
\end{equation}%
where the last integral makes sense as Young integral since the (continuous)
integrand has finite $p$-variation regularity. Noting that the
non-degeneracy condition on $\left[ 0,T\right] $ implies the same
non-degeneracy condition on $\left[ 0,t\right] $ we see that the integrand
in (\ref{vDY_is_zero}) must be zero on $\left[ 0,t\right] $ and evaluation
at time $0$ shows that for all $j=1,...,d$,%
\begin{equation*}
v^{T}J_{t\leftarrow 0}^{\mathbf{X}}\left( V_{j}\left( y_{0}\right) \right)
=0.
\end{equation*}%
It follows that the vector $v^{T}J_{t\leftarrow 0}^{\mathbf{X}}$ is
orthogonal to $V_{j}\left( y_{0}\right) ,\,\,\,j=1,...,d$ and hence zero.
Since $J_{t\leftarrow 0}^{\mathbf{X}}$ is invertible we see that $v=0$. The
proof is finished.
\end{proof}

The reader may be curious to hear about smoothness in this context. Adapting
standard arguments would require $L^{p}\left( \Omega \right) $ estimates on
the Jacobian of the flow $J_{t\leftarrow 0}^{\mathbf{X}\left( \omega \right)
}$. Using the fact that it satisfies a linear RDE, $dJ_{t\leftarrow t_{0}}^{%
\mathbf{X}}\left( y_{0}\right) =d\mathbf{M}^{\mathbf{X}}\cdot J_{t\leftarrow
t_{0}}^{\mathbf{X}}\left( y_{0}\right) $, with $d\mathbf{M}^{\mathbf{X}%
}=V^{\prime }\left( Y\right) d\mathbf{X}$ one can see that%
\begin{equation}
\log \left\vert J_{t\leftarrow t_{0}}^{\mathbf{X}}\left( y_{0}\right)
\right\vert =O\left( \left\Vert \mathbf{X}\right\Vert _{p\text{-var}%
}^{p}\right) .  \label{logJestimates}
\end{equation}%
(This estimate appears in \cite{NuHu} for $p<2$ but can be seen \cite{L98,
FVbook} to hold for all $p\geq 1$. We believe it to be optimal.) Using the
Gauss tail of the homogenous $p$-variation norm of Gaussian rough paths (see 
\cite{FV2, FV07}) we see that $L^{q}$-estimates for all $q<\infty $ hold
true when $p<2$ and this underlies to density results of \cite{NuHu, BH}. On
the other hand, for $p>2$ one cannot obtain $L^{q}$-estimates from (\ref%
{logJestimates}) and further probabilistic input will be needed.


\begin{thebibliography}{99}
\bibitem{BH} Baudoin, F.; Hairer, M.: Hoermander's theorem for fractional
Brownian motion, to appear in Probab. Theory Relat. Fields (2007).

\bibitem{BaHoSi} Bayraktar,\ E.; Horst, U.; Sircar R.: A Limit Theorem for
Financial Markets with Inert Investors, Mathematics of Operations Vol 31,
2006

\bibitem{Be} Bell, D.: The Malliavin calculus. Dover Publications, Inc.,
Mineola, NY, 1987, 2006

\bibitem{Bi} Bismut, J.-M.: Large deviations and the Malliavin calculus.
Progress in Mathematics, 45. Birkh\"{a}user Boston, Inc., Boston, MA, 1984

\bibitem{FLS} Friz, P., Lyons, T., Stroock, D.: L\'{e}vy's area under
conditioning. Ann. Inst. H. Poincar\'{e} Probab. Statist. 42 (2006), no. 1,
89--101

\bibitem{FV1} Friz,\ P., Victoir, N.: Approximations of the Brownian rough
path with applications to stochastic analysis. Ann. Inst. H. Poincar\'{e}
Probab. Statist. 41 (2005), no. 4, 703--724

\bibitem{FV2} Friz,\ P., Victoir, N.: A note on the notion of geometric
rough paths. Probab. Theory Related Fields 136 (2006), no. 3, 395--416

\bibitem{FV07} Friz,\ P., Victoir, N.: Differential Equations Driven by
Gaussian Signals I. arXiv-preprint.

\bibitem{FVbook} Friz,\ P., Victoir, N.: Multidimensional Stochastic
Processes as Rough Paths. Theory and Applications, Cambridge University
Press (in preparation).

\bibitem{Gu06} Guasoni, P.: No Arbitrage with Transaction Costs with
Fractional Brownian Motion and Beyond, Mathematical Finance, 16 (2006).

\bibitem{Gu07} Guasoni, P.; Rasonyi, M.; Schachermayer, W.: The Fundamental
Theorem of Asset Pricing for Continuous Processes under Small Transaction
Costs. Preprint (2007)

\bibitem{Ha05} Hairer, M.: Ergodicity of stochastic differential equations
driven by fractional Brownian motion, Ann. Probab. 33 (2005), no 3, pp.
703-758.

\bibitem{Ku82} Kusuoka, S.: The nonlinear transformation of Gaussian measure
on Banach space and absolute continuity. I. J. Fac. Sci. Univ. Tokyo Sect.
IA Math. 29 (1982), no. 3, 567--597.

\bibitem{Ku82b} Kusuoka, S.: Dirichlet forms and diffusion processes on
Banach space. J. Fac. Sci. Univ. Tokyo Sect. IA Math. 29-1 (1982).

\bibitem{Ku} Kusuoka, S.: On the regularity of solutions to SDEs.`Asymptotic
problem in probability theory: Wiener functionals and asymptotics' ed. K. D.
Elworthy and N. Ikeda, Pitman Res. Notes Math. Ser., 284, pp.90--106,
Longman Scientific \& Technica, 1993

\bibitem{KS82} Kusuoka, S.; Stroock, D.: Applications of the Malliavin
calculus I. Proceedings of the Taniguchi International Symposium on
Stochastic Analysis, Katata and Kyoto, 1982. North-Holland.

\bibitem{KS87} Kusuoka, S.; Stroock, D.: Applications of the Malliavin
calculus III. J. Fac. Sci. Univ. Tokyo Sect. IA Math. 34 (1987), no. 2,
391--442

\bibitem{L98} Lyons, T.: Differential equations driven by rough signals,
Rev. Mat. Iberoamericana 14 (1998), no. 2, 215--310.

\bibitem{LQ97} Lyons, T.; Qian, Z.: Calculus of variation for multiplicative
functionals, New trends in stochastic analysis (Charingworth, 1994),
348-374, World Sci. Publishing, River Edge, NJ, 1997

\bibitem{LQ98} Lyons, T.; Qian, Z.: Qian, Zhongmin Flow of diffeomorphisms
induced by a geometric multiplicative functional. Probab. Theory Related
Fields 112 (1998), no. 1, 91--119

\bibitem{LQ02} Lyons, T.; Qian, Z.: System Control and Rough Paths,\ Oxford
University Press, 2002.

\bibitem{Ma} Malliavin, P.: Stochastic analysis. Grundlehren der
Mathematischen Wissenschaften [Fundamental Principles of Mathematical
Sciences], 313. Springer-Verlag, Berlin, 1997

\bibitem{McK} McKean, H. P.: Stochastic Integrals, Academic Press, New
York-London (1969).

\bibitem{Nu06} Nualart, D.: The Malliavin calculus and related topics.
Second edition. Probability and its Applications (New York).
Springer-Verlag, Berlin, 2006

\bibitem{NuSa} Nualart, D.; Saussereau, B.: Malliavin calculus for
stochastic differential equations driven by a fractional Brownian motion.
Preprint\ (2005).

\bibitem{NuHu} Hu, Y.; Nualart,\ D.: \ \ Differential equations driven by
Hoelder continuous functions of order greater than 1/2; ArXiv
(math.PR/0601628)

\bibitem{Su85} Sugita, H.: On a characterization of the Sobolev spaces over
an abstract Wiener Space. J. Fac. Sci. Univ. Tokyo Sect. IA Math. 25-1
(1985) 31-48.

\bibitem{Sh} Shigekawa, I.: Stochastic analysis. Translated from the 1998
Japanese original by the author. Translations of Mathematical Monographs,
224. Iwanami Series in Modern Mathematics. American Mathematical Society,
Providence, RI, 2004.

\bibitem{To} Towghi, Nasser: Multidimensional extension of L. C. Young's
inequality. JIPAM. J. Inequal. Pure Appl. Math., 3(2):Article 22, 13 pp.
(electronic), 2002

\bibitem{UZ00} Uestuenel, A. S.; Zakai, M.: Transformation of measure on
Wiener space. Springer-Verlag, Berlin, 2000

\bibitem{Yo} Young, L.C.: An inequality of H\"{o}lder type, connected with
Stielties integration, Acta Math. 67, 251-282, 1936
\end{thebibliography}
\end{document}